\documentclass[a4paper,12pt]{amsart}
\usepackage{amsfonts}
\usepackage{amsthm}
\usepackage{amssymb}
\usepackage{amsmath}
\usepackage{enumerate}
\usepackage{url}

\usepackage{a4wide}

\begin{document}
\title{Decoupling theorems for the Duffin-Schaeffer problem}
\author{Christoph Aistleitner}

\begin{abstract}
The Duffin-Schaeffer conjecture is a central open problem in metric number theory. Let $\psi~\mathbb{N} \mapsto \mathbb{R}$ be a non-negative function, and set $\mathcal{E}_n :=\bigcup \left( \frac{a - \psi(n)}{n},\frac{a+\psi(n)}{n} \right)$, where the union is taken over all $a \in \{1, \dots, n\}$ which are co-prime to $n$. Then the conjecture asserts that almost all $x \in [0,1]$ are contained in infinitely many sets $\mathcal{E}_n$,  provided that the series of the measures of $\mathcal{E}_n$ is divergent. At the core of the conjecture is the problem of controlling the measure of the pairwise overlaps $\mathcal{E}_m \cap \mathcal{E}_n$, in dependence on $m, n, \psi(m)$ and $\psi(n)$. In the present paper we prove upper bounds for the measures of these overlaps, which show that globally the degree of dependence in the set system $(\mathcal{E}_n)_{n \geq 1}$ is significantly smaller than supposed. As applications, we obtain significantly improved ``extra divergence'' and ``slow divergence'' variants of the Duffin-Schaeffer conjecture.
\end{abstract}
\maketitle

\newtheorem{prop}{Proposition}
\newtheorem{claim}{Claim}
\newtheorem{lemma}{Lemma}
\newtheorem{thm}{Theorem}
\newtheorem{defn}{Definition}
\newtheorem{conj}{Conjecture}

\theoremstyle{definition}
\newtheorem{exmp}{Example}

\theoremstyle{remark}
\newtheorem{rmk}{Remark}

\section{Introduction and statement of results}

\emph{Opening remark: The results obtained in this manuscript have been superseded by those of Koukoulopoulos and Maynard \cite{km}, who gave a proof of the full Duffin--Schaeffer conjecture. This manuscript is placed on arxiv for reference purpose, but will not be published in a mathematical journal. It is left in the (unpolished) form which it had when I learned about the existence of Koukoulopoulos' and Maynard's proof, except for the addition of this opening remark and a closing remark at the end of the introduction.}\\

Let $\psi: \mathbb{N} \rightarrow \mathbb{R}$ be a non-negative function. For every non-negative integer $n$ define a set $\mathcal{E}_n \subset \mathbb{R} / \mathbb{Z}$ by 
\begin{equation} \label{edef}
\mathcal{E}_n := \bigcup_{\substack{1 \leq a \leq n,\\ (a,n)=1}} \left( \frac{a - \psi(n)}{n},\frac{a+\psi(n)}{n} \right).
\end{equation}
The Lebesgue measure of $\mathcal{E}_n$ is at most $2 \psi(n) \varphi(n)/n$, where $\varphi$ denotes the Euler totient function. Thus, writing $W(\psi)$ for the set of those $x \in [0,1]$ which are contained in infinitely many sets $\mathcal{E}_n$, it follows directly from the first Borel--Cantelli lemma that $\lambda(W(\psi))=0$ whenever
\begin{equation} \label{sum}
\sum_{n=1}^\infty \frac{\psi(n) \varphi(n)}{n} < \infty.
\end{equation}
Here $\lambda$ denotes the Lebesgue measure. The corresponding divergence statement, which asserts that $\lambda(W(\psi))=1$ whenever the series in \eqref{sum} is divergent, is known as the Duffin--Schaeffer conjecture \cite{ds} and is one of the most important open problems in metric number theory. It remains unsolved since 1941.\\

Historically, the Duffin--Schaeffer conjecture is an attempt to remove the monotonicty requirement Tfrom Khintchine's theorem in metric Diophantine approximation. The Duffin--Schaeffer conjecture is known to be true under some additional arithmetic conditions or regularity conditions on the function $\psi$. A basic result, known as the Duffin--Schaeffer \emph{theorem}, asserts that the conclusion of the conjecture holds whenever the additional assumption
\begin{equation} \label{dst}
\limsup_{N \to \infty} \frac{\sum_{n=1}^N \frac{\psi(n) \varphi(n)}{n}} {\sum_{n=1}^N \psi(n)} > 0
\end{equation}
is satisfied. Further results under assumptions on the arithmetic properties of the support of $\psi$ were obtained by Harman \cite{har2} and by Strauch, in a series of papers starting with \cite{strauch}. One of the most striking results is the Erd\H os--Vaaler theorem \cite{vaa}, which states that the conclusion of the conjecture holds under the assumption that $\psi(n) \leq 1/n$ for all $n$ (without imposing any further arithmetic conditions). This has been slightly improved later by Vilchinskii \cite{vil}. For more basic information on the problem and an exposition of classical results, see Harman's \cite{harman} monograph on Metric Number Theory.\\

Observe that the second Borel--Cantelli lemma cannot be used to deduce the conclusion of the conjecture from the divergence of the series \eqref{sum}, since the sets $(\mathcal{E}_n)_{n \geq 1}$ are not independent. Indeed, it is well-known that by the Erd\H os--R\'enyi version of the Borel--Cantelli lemma (see Lemma \ref{lemmabc} below), together with Gallagher's zero-one law \cite{gall}, it would be sufficient to establish pairwise ``quasi-independence on average'' of these sets. However, the best that we have is the following estimate of Pollington and Vaughan \cite{pv}.\footnote{Concerning the notation: Throughout this paper we will understand $\log x$ as $\max(1,\log x)$, so that all appearing logarithms and iterated logarithms are positive and well-defined. We use Vinogradov notation $\ll$ and $\gg$, where any dependence of the implied constants will be indicated. We will write $\eta$ and $\varepsilon$ for fixed constants which can be chosen arbitrarily small.}
\begin{lemma} \label{lemmapv}
Assume that $m < n$. Define \begin{equation} \label{ddef}
D(m,n) = \frac{\max(n \psi(m),m \psi(n))}{(m,n)}.
\end{equation}
When $D < 1$, then $\mathcal{E}_m \cap \mathcal{E}_n = \emptyset$. When $D \geq 1$, then  
\begin{equation} \label{lemmaint}
\lambda(\mathcal{E}_m \cap \mathcal{E}_n) \leq P(m,n) \lambda(\mathcal{E}_m) \lambda(\mathcal{E}_n),
\end{equation}
where
\begin{equation} \label{pdef}
P(m,n) \ll \prod_{\substack{p | \frac{mn}{(m,n)^2},\\p > D(m,n)}} \left(1 - \frac{1}{p} \right)^{-1}.
\end{equation}
\end{lemma}

Two things are crucial here. On the one hand, the factor $P(m,n)$ in the lemma is unbounded, and can be of order as large as $\log \log n$. On the other hand, this can only happen when $D(m,n)$ is in a ``critical range'' from $1$ to $(\log n)^\eta$ for some positive $\eta$, since it can be shown that $P(m,n) \ll_\eta 1$ whenever $D \gg (\log n)^\eta$. It should be noted that the problem with estimating the measure of the overlaps $\mathcal{E}_m \cap \mathcal{E}_n$ is \emph{not} that we are missing good estimates for these measures; on the very contrary, morally Lemma \ref{lemmapv} can be treated as an equality. Thus for some configurations of $m,n, \psi(m), \psi(n)$ the measure of the overlap $\mathcal{E}_m \cap \mathcal{E}_n$ \emph{really} is too large. Quoting from \cite{extra}: 
\begin{quote}
This is a real problem, not just a deficiency in our knowledge.
\end{quote}
In many partial results, the factor $P(m,n)$ is controlled by imposing arithmetic conditions upon the support of $\psi$. The Duffin--Schaeffer theorem might be seen in this light, since $P(m,n)$ can be estimated in terms of the Euler totient function of $m$ and $n$, and can thus be controlled using \eqref{dst}. More delicately (and more recently), in \cite{a2} and \cite{extra} an ``extra divergence'' assumption was used to shift $\psi$ such that the critical range for $D(m,n)$ can be avoided. In these papers it was tried to control $D(m,n)$ and $P(m,n)$ on an individual basis, that is, for specific pairs of $m$ and $n$. In the present paper we take a very different, ``global'' perspective, which is more in the spirit of \cite{vaa}. We show that even if for some configurations of $m,n,\psi(m)$ and $\psi(n)$ the value of $D(m,n)$ may fall into the critical range and the factor $P(m,n)$ may be too large, under certain circumstances this can only happen for a number of pairs of indicies $m$ and $n$ which is negligible from a global perspective. This approach is in accordance with the following sentence, which is the direct continuation in \cite{extra} of the quotation above:
\begin{quote}
Our hope would be that the values of $m$ and $n$ concerned do not make the major contribution to
$$
\sum_{1 \leq m,n \leq N} \lambda ( \mathcal{E}_m \cap \mathcal{E}_n ). 
$$
\end{quote}

The structural results of this paper are of a somewhat technical natural; they are formulated as Lemmas in the following section. Here in the introduction we will only illustrate the quantitative improvements coming from these lemmas, to show how they imply that there is much less structural dependence in the Duffin--Schaeffer problem than what usually was assumed so far. Subsequently, we present two applications, concerning improvements of recent work on ``extra divergence'' and ``slow divergence'' versions of the Duffin--Schaeffer problem. We finish the introduction with a short survey on certain sums involving greatest common divisors (GCD sums), which play a key role in our proofs.\\

Throughout the following statements, assume that $m < n$. \\
\begin{itemize}
\item It was know that, as a consequence of the Erd\H os--Vaaler theorem and Lemma \ref{lemmapv}, we have $\lambda(\mathcal{E}_m \cap \mathcal{E}_n) \ll \lambda(\mathcal{E}_m) \lambda(\mathcal{E}_n)$, provided that $m^4 \leq n$. See for example \cite{extra}. Our results show that actually it is sufficient to assume that $m (\log m)^\eta \ll n$, for some $\eta>0$ --- not for individual pairs of indices $m$ and $n$, but globally in the sense that the number of exceptional pairs of indices $m$ and $n$ is negligible. This allows us to localize the problem with respect to the relative size of $m$ and $n$.\\
\item Similarly, we show that $\lambda(\mathcal{E}_m \cap \mathcal{E}_n) \ll \lambda(\mathcal{E}_m) \lambda(\mathcal{E}_n)$ whenever either $\psi(m) (\log n)^\eta \leq \psi(n)$ or $\psi(m) \geq \psi(n) (\log n)^\eta$, for some $\eta > 0$ ---  again in the sense that the number of exceptional pairs of indices $m$ and $n$ is negligible. This allows us to localize the problem with respect to the relative position of $\psi(m)$ and $\psi(n)$.\\
\item It was known that $P(m,n)\ll (\log \log n)$ for all $m$ and $n$. We show that actually we always have $P(m,n) \ll (\log \log \log n)$, except for a number of pairs $m$ and $n$ which is negligible. 
\end{itemize}

\subsection{Extra divergence}

In \cite{extra1}, Haynes, Pollington and Velani initiated a program to establish the Duffin--Schaeffer condition without assuming any regularity properties or arithmetic properties of $\psi$, but instead assuming a slightly stronger divergence condition. In \cite{extra1} they proved that there is a constant $c$ such that $\lambda(W(\psi))=1$, provided that
$$
\sum_{n=1}^\infty \frac{\psi(n) \varphi(n)}{n ~ e^{\left( \frac{c \log n}{\log \log n} \right)}} = \infty
$$
Beresnevich, Harman, Haynes and Velani \cite{extra} developed a beautiful averaging argument to show that it is sufficient to assume 
$$
\sum_{n=1}^\infty \frac{\psi(n) \varphi(n)}{n(\log n)^{\varepsilon \log \log \log n}} = \infty
$$
for some $\varepsilon > 0$. Using a more subtle version of their argument, in \cite{a2} the extra divergence requirement was reduced to
\begin{equation} \label{extra_th}
\sum_{n=1}^\infty \frac{\psi(n) \varphi(n)}{n(\log n)^{\varepsilon}} = \infty
\end{equation}
for some $\varepsilon > 0$.\\

In the present paper we obtain the following ``extra divergence'' result.
\begin{thm} \label{th1}
Let $\psi: \mathbb{N} \rightarrow [0,\infty)$ be a function. Assume that there is a constant $\varepsilon >0$ such that
\begin{equation} \label{div}
\sum_{n=1}^\infty \frac{\psi(n) \varphi(n)}{n(\log \log n)^\varepsilon} = \infty.
\end{equation}
Then we have $\lambda(W(\psi))=1$.
\end{thm}

Reducing the ``extra divergence'' factor to a power of $\log \log n$ is psychologically significant, since this is the scale where the factor $\varphi(n)/n$ becomes visible in the extra divergence statement. Indeed, since $1 \geq \varphi(n) /n \gg (\log \log n)^{-1}$ for all $n$, rather than assuming \eqref{extra_th} for some $\varepsilon>0$ we could also assume that
$$
\sum_{n=1}^\infty \frac{\psi(n)}{(\log n)^{\hat{\varepsilon}}} = \infty
$$
for some $\hat{\varepsilon} > 0$. Theorem \ref{th1} does not have such a simple equivalent formulation without the Euler totient function.

\subsection{Slow divergence}

In \cite{aistslow} the author proved the following ``slow divergence'' variant of the Duffin--Schaeffer conjecture: The conclusion of the Duffin--Schaeffer conjecture holds, provided that 
\begin{equation} \label{series_div}
\sum_{n=1}^\infty \frac{\psi(n) \varphi(n)}{n} = \infty
\end{equation}
and 
\begin{equation} \label{sum_h}
\sum_{2^{2^h} < n \leq 2^{2^{h+1}}} \psi(n) \ll \frac{1}{h}.
\end{equation}
The purpose of this result was to show that in any potential counterexample to the Duffin--Schaeffer conjecture, the mass of $\psi$ must be unevenly distributed over the positive integers. Indeed, if $\psi$ is ``regular'' (using the word in a completely informal sense), then we should expect the sum on the left-hand side of \eqref{sum_h} to be somewhere around $1/(h \log h)$, since this is the range where the convergence/divergence of the series \eqref{series_div} is decided. As a consequence of our decoupling results, we obtain the following drastically improved ``slow divergence'' theorem.

\begin{thm} \label{th2}
Let $\psi: \mathbb{N} \rightarrow [0,\infty)$ be a function. Assume that the divergence requirement \eqref{series_div} holds. Assume additionally that there exists a constant $\eta > 0$ such that
$$
\sum_{2^h < n \leq 2^h h^\eta} \frac{\psi(n) \varphi(n)}{n} \ll \frac{1}{\log \log h}.
$$
Then $\lambda(W(\psi))=1$. 
\end{thm}

Actually, we can also include a restriction on the size of $\psi$. Then the theorem reads as follows. Note that Theorem \ref{th2} is a direct consequence of Theorem \ref{th3}.

\begin{thm} \label{th3}
Let $\psi: \mathbb{N} \rightarrow [0,\infty)$ be a function. Assume that the divergence requirement \eqref{series_div} holds. Assume additionally that there exists a constant $\eta > 0$ such that
$$
\sum_{\substack{2^h < n \leq 2^h h^\eta, \\ r^{-1} h^{-\eta} < \psi(n) \leq r^{-1}}} \frac{\psi(n) \varphi(n)}{n} \ll \frac{1}{\log \log h},
$$
uniformly in $r \geq 1$. Then $\lambda(W(\psi))=1$. 
\end{thm}

Note that Theorem \ref{th2} improves the earlier ``slow divergence'' result in two directions. On the one hand, the summation range is reduced from double exponential to slightly more than exponential (you may take a moment to convince yourself that shortening the summation range indeed is an improvement). On the other hand, the required upper bound for the block sums is significantly weaker. As noted before, in a ``regular'' function $\psi$ the critical region for the block sums should be near $1/(h \log h)$ -- in Theorem \ref{th2} instead we require the upper bound $1/(\log \log h)$ for such block sums. Thus the conclusion of Theorem \ref{th2} can only fail if the mass of $\psi$ is extremely unevenly distributed. For many applications it should be possible to rule out such an extremely uneven distribution of the mass of the approximation function.

\subsection{GCD sums}

A GCD sum is a sum of the form
\begin{equation} \label{gcd_s}
\sum_{1 \leq k,\ell \leq N} \frac{(n_k,n_\ell)^{2 \alpha}}{(n_k n_\ell)^\alpha} \qquad \text{(without coefficients)}
\end{equation}
or
\begin{equation} \label{gcd_s2}
\sum_{1 \leq k,\ell \leq N} c_k c_\ell \frac{(n_k,n_\ell)^{2 \alpha}}{(n_k n_\ell)^\alpha} \qquad \text{(with coefficients)}.
\end{equation}
Here $\{n_1, \dots, n_N\}$ are distinct positive integers and $\alpha$ is a real parameter, usually from the range $[1/2,1]$. In the case of coefficients, the problem is normalized by assuming that $\sum c_k^2 = 1$. The most interesting problem for such sums is to find general upper bounds for \eqref{gcd_s} and \eqref{gcd_s2} which depend only on $N$, but not on the choice of $n_1, \dots, n_N$ or on the coefficients $c_1, \dots, c_N$.\\

It seems that such sums were first considered in the 1920s or 1930s by Erd\H os and Koksma in the context of Diophantine approximation, with the parameter $\alpha=1$. They realized that GCD sums can be used to give upper bounds for square-integrals (that is, variances) of sums of dilated functions; see \cite{koks} for an early reference. Actually, for a specific choice of the function there even is an equality; an example of such a relation is Franel's identity, which states that
\begin{equation} \label{franel}
\int_0^1 \left(\sum_{k=1}^N c_k (\{n_k x\} - 1/2)\right)^2 dx = \frac{\pi^2}{12} \sum_{k,\ell=1}^N c_k c_\ell \frac{(\gcd(n_k, n_\ell))^2}{n_k n_\ell},
\end{equation}
where $\{ \cdot \}$ denotes the fractional part. The problem of bounding GCD sums can also be seen in terms of bounding the maximal eigenvalue of certain symmetric matrices containing greatest common divisors -- this approach might have its first appearance in work of Wintner \cite{wintner} in 1944. Remarkably, the GCD sum can also be realized as an integral involving the Riemann zeta function, along a vertical line in the complex plane -- this is the viewpoint taken in \cite{l_r}. See \cite{a1} for a more detailed presentation of some of these connections.\\

The problem of finding the maximal asymptotic order of \eqref{gcd_s} in the case $\alpha=1$ was posed by Erd\H os, and solved by G\'al \cite{gal} in 1949. In the case $\alpha=1/2$, partial results were obtained by Dyer and Harman \cite{d_h} in 1986; these were applied by Harman \cite{har3,har2} to establish some special cases of the Duffin--Schaeffer conjecture.\\

In recent years there has been increased interest in GCD sums, and optimal bounds for the maximal order of \eqref{gcd_s} and \eqref{gcd_s2} have been established in all remaining cases. The case $\alpha=1$ in the situation with coefficients was solved by Lewko and Radziwi{\l}{\l} \cite{l_r}. The case $\alpha \in (1/2,1)$ was solved in \cite{a_seip}, and the case $\alpha = 1/2$ was solved in \cite{bond_s}. There is a ``phase transition'' in the behavior of the maximal order of the GCD sum with respect to the parameter $\alpha$, which is mirrored by a similar transition of the behavior of the zeta function $\zeta(\sigma+it)$ in the critical strip with respect to $\sigma$. In terms of metric number theory, GCD sums with parameter $\alpha=1$ are usually associated with sums of dilated function where the function is fixed as in \eqref{franel} or as in the convergence problems in \cite{a_seip,l_r}, while GCD sums with $\alpha=1/2$ correspond to ``shrinking targets'' such as sums of indicator functions of short intervals in metric Diophantine approximation, or as in the related context of pair correlations of parametric sequences (see for example \cite{all,walker,rud}). \\

The case $\alpha \in (0,1/2)$ seems to be much less natural. In this range the connection with the Riemann zeta function breaks down \cite{hilb}. Similarly, the connection with sums of dilated functions breaks down, since the corresponding function would not be in $L^2$ anymore. However, quite remarkable, it is this range of parameter which we use in the present paper, since it leads to the strongest results. It seems that this is the first time that GCD sums with parameter $\alpha$ smaller than 1/2 have been applied in a number-theoretic problem.\\

Very roughly speaking, the connection of the Duffin--Schaeffer problem with GCD sums is the following. As noted above, the overlap $\mathcal{E}_m \cap \mathcal{E}_n$ can only be too large when $D(m,n)$, as defined in \eqref{ddef}, lies in some critical range. Note that in $D(m,n)$ there is an explicit dependence on the GCD of $m$ and $n$. One can check that $D(m,n)$ can only be in the critical range when $D(m,n)$ is ``large'' in some appropriate sense. However, an upper bound for the GCD sum directly implies a bound for the number of pairs of indices for which the GCD can be large. Note that an argument of this type does not address the potential size of the overlaps for individual pairs of indices as in \cite{a2,extra}, but rather assesses the potential behavior of these overlaps on a global scale; in this sense our argument is much more in the spirit of the one in the proof of the Erd\H os--Vaaler theorem.\\

It turns out that the estimates for GCD sums only apply when we can assure that either $m$ and $n$, or that $\psi(m)$ and $\psi(n)$ differ in order by a logarithmic factor. To exploit this phenomenon we establish a sort of Cauchy--Schwarz inequality for GCD sums (Lemma \ref{lemma_h2}). What happens is that when $n$ moves away from $m$ (or when $\psi(n)$ moves away from $\psi(m)$), the GCD of $m$ and $n$ would need to grow linearly in $n/m$ to keep $D(m,n)$ in the critical range; however, our Cauchy--Schwarz inequality only allows the GCDs to grow proportional with $\sqrt{n/m}$, with the exception of a negligible set of pairs of indices. In the case when we cannot guarantee that $m$ and $n$ (or $\psi(m)$ and $\psi(n)$) are of different order, we introduce a sum-of-distinct-prime-divisors function into the GCD sum (Lemma \ref{lemma_h3}). Bounding the number of distinct prime divisors of $m$ and $n$ allows us to give an upper bound for the factor $P(m,n)$, which was defined in \eqref{pdef}.\\

It is not clear if GCD sums are the ``correct'' tool to exploit the phenomena that we observed above. It is probably difficult to estimate directly the number of pairs of indices $m$ and $n$ for which $D(m,n)$ can lie in the critical range, and the corresponding maximal size of $P(m,n)$. The situation becomes easier by translating the problem into a problem involving GCD sums, since it can be shown that such sums are maximized by sets of integers which have a very strong multiplicative structure, and for such special sets the GCD sum can be efficiently evaluated. For further improvements, it seems that one would not only have to estimate the size of the greatest common divisors or the number of distinct prime divisors involved, but rather to determine the structure of the set of greatest common divisors themselves. Morally speaking, one might hope that in a set of integers where pairwise greatest common divisors are very large, there should be a \emph{common} large factor which appears in the factorization of all these integers. At such a point, one could hope to discard this common large factor and to exploit the same phenomena as in the Erd\H os--Vaaler theorem in an ``uplifted'' setting.\\

\emph{Closing remark: The strategy sketched in the previous paragraph is essentially the one which is used in Koukoulopoulos' and Maynard's proof. Instead of working with GCD sums, they introduce a much more subtle structure which they call ``GCD graph'', on which they perform a ``descend'' along ``GCD subgraphs'' towards a setting where they can single out a large common divisor and apply a variant of the Erd\H os--Vaaler argument. While the GCD sum can only control the size of common divosors, the GCD graph can also control structural properties of the divisor system. A trace of the descent along the GCD subgraphs in the K-M argument can be found in the way how upper bounds for GCD sums are proved by a transition towards the worst-case (divisor-closed, square-free, etc.) scenario.}

\section{Auxiliary results}

We will use Lemma \ref{lemmapv}. As noted, it is well-known that for $m < n$
\begin{equation} \label{pmn}
P(m,n) \ll_\eta 1 \qquad \text{if} \qquad D(m,n) \geq (\log n)^{\eta},
\end{equation}
for any $\eta >0$. Furthermore, the factor $P(m,n)$ is of order at most $\log \log n$, and thus
\begin{equation} \label{prod_ratio}
\lambda(\mathcal{E}_m \cap \mathcal{E}_n) \ll \lambda(\mathcal{E}_m) \lambda(\mathcal{E}_n) \log \log n.
\end{equation}
Both facts follow easily from Mertens' theorems. As a reference, see for example the first formula on p.~132 of \cite{extra}.\\

We will use the following version of the second Borel--Cantelli lemma (see for example \cite[Lemma 2.3]{harman}).
\begin{lemma} \label{lemmabc}
Let $\mathcal{A}_n,~n =1,2,\dots$, be events in a probability space $(\Omega,\mathcal{F},\mathbb{P})$. Let $\mathcal{A}$ be the set of $\omega \in \Omega$ which are contained in infinitely many $\mathcal{A}_n$. Assume that
$$
\sum_{n=1}^\infty \mathbb{P}(\mathcal{A}_n) = \infty.
$$
Then
$$
\mathbb{P} (\mathcal{A}) \geq \limsup_{N \to \infty} \frac{\left( \sum_{n=1}^N \mathbb{P}(\mathcal{A}_n) \right)^2}{\sum_{1 \leq m,n \leq N} \mathbb{P} (\mathcal{A}_m \cap \mathcal{A}_n)}.
$$
\end{lemma}

For positive integers $r$ and $r$ we define
\begin{equation} \label{skr_def}
S_k^r := \Big\{2^k < n \leq 2^{k+1}:~ \psi(n) \in [2^{-r},2^{-r-1}] \Big\}.
\end{equation}

It is well-known that in the Duffin-Schaeffer conjecture we can assume that $1/n \leq \psi(n) \leq 1/2$ for all $n$. The first inequality is the Erd\H{o}s--Vaaler theorem \cite{vaa}, the second inequality is in \cite{pv}. Thus, throughout this paper, in the decomposition into sets $S_k^r$ we can always assume that $r \leq k$. Furthermore, we may also assume throughout the paper that 
\begin{equation} \label{skr_size}
\# S_k^r \leq k 2^r
\end{equation}
for all $k$ and $r$. Indeed, assume on the contrary that $\# S_k^r \geq k 2^r$ for infinitely many pairs of $k$ and $r$. It is easily see that this implies that
$$
\sum_{n \in S_k^r} \frac{\lambda(\mathcal{E}_n)}{\log n} \gg 1
$$
for such pairs $k$ and $r$, a situation in which the extra divergence result \eqref{extra_th} applies. Thus we may assume throughout the rest of this paper that \eqref{skr_size} holds. Furthermore, we may also assume throughout this paper that
\begin{equation} \label{skr_lower}
\# S_k^r \geq \frac{2^r}{k^2},
\end{equation}
since otherwise 
$$
\sum_{n \in S_k^r} \lambda(\mathcal{E}_n) \ll \frac{1}{k^2}, 
$$
and accordingly the integers in $S_k^r$ do not contribute to the divergence of the series \eqref{div}, and we may completely remove them from the support of $\psi$.\\

The following lemma is \cite[Theorem 1]{hilb}, in the special case $\alpha=1/4$. 

\begin{lemma} \label{lemma_h1}
There exists a constant $b_1>0$ such that the following holds. Let $\mathcal{M}$ denote a finite set of distinct positive integers, and write $N = \# \mathcal{M}$. Then
$$
\sum_{m,n \in \mathcal{M}} \frac{(m,n)^{1/2}}{(mn)^{1/4}} \leq N^{3/2} (\log N)^{b_1},
$$
provided that $N$ is sufficiently large.
\end{lemma}

\begin{lemma} \label{lemma_h2}
There exists a constant $b_2>0$ such that the following holds. Let $\mathcal{M}_1$ and $\mathcal{M}_2$ denote two finite set of distinct positive integers, and write $N_1 = \# \mathcal{M}_1$ and $N_2 = \# \mathcal{M}_2$. Assume w.l.o.g.\ that $N_1 \leq N_2$. Assume furthermore that $\mathcal{M}_1 \subset \{2^{k_1},2^{k_1+1}\}$ and $\mathcal{M}_2 \subset \{2^{k_2},2^{k_2+1}\}$ for some positive integers $k_1,k_2$. Let $R>1$ be a real number. Then
$$
\sum_{\substack{m \in \mathcal{M}_1,n \in \mathcal{M}_2,\\ R \leq (m,n) \leq R \log N_2}} \frac{(m,n)^{1/2}}{(mn)^{1/4}} \leq (N_1 N_2)^{3/4} (\log N_2)^{b_2},
$$
provided that $N_1$ and $N_2$ are sufficiently large.
\end{lemma}

\begin{lemma} \label{lemmak1k2}
Let $S_{k_1}^{r_1}$ and $S_{k_2}^{r_2}$ be two sets as defined above, such that $k_1 \geq k_2$. Let $b_3$ be a positive constant such that $\frac{b_3 (\log 2)}{8} > b_2 + 4$, where $b_2$ is the constant from the previous lemma. Assume that either 
\begin{equation} \label{lemmak1k2ass}
|k_1 - k_2| \geq b_3 \log k_1 \qquad \textup{or} \qquad |r_1 - r_2| \geq b_3 \log k_1.
\end{equation}
Then
$$
\sum_{m \in S_{k_1}^{r_1}} \sum_{n \in S_{k_2}^{r_2}} \lambda(\mathcal{E}_m \cap \mathcal{E}_n) \ll \left(\sum_{m \in S_{k_1}^{r_1}} \lambda(\mathcal{E}_m)\right) \left(\sum_{n \in S_{k_2}^{r_2}} \lambda(\mathcal{E}_n)\right).
$$
\end{lemma}

Informally speaking, Lemma \ref{lemmak1k2} says the following. Let $S_{k_1}^{r_1}$ and $S_{k_2}^{r_2}$ be two sets of indices as above. Assume that both $S_{k_1}^{r_1}$ and $S_{k_2}^{r_2}$ are so large that they contribute to the divergence of the series \eqref{div}. Then the set systems ${E}_m, ~m \in S_{k_1}^{r_1}$ and $E_n, ~ n \in S_{k_2}^{r_2}$ are essentially independent, provided that either the elements of $S_{k_1}^{r_1}$ and $S_{k_2}^{r_2}$, or the individual measures assigned to these elements, are of significantly different order. Here ``significantly different'' means multiplication by a power of logarithm. For example, the lemma applies if all the elements of $S_{k_1}^{r_1}$ are of order roughly $N$ for some $N$, and the elements of $S_{k_2}^{r_2}$ are of order at least $N (\log N)^{100}$. Similarly, the lemma applied if there is a factor of the size of a power of $\log$ separating the order of $\psi(m),~m \in S_{k_1}^{r_1}$ and $\psi(n),~n \in S_{k_2}^{r_2}$.\\

Note that a statement like Lemma \ref{lemmak1k2} is \emph{not} true for individual sets $\mathcal{E}_m$ and $\mathcal{E}_n$. The machinery of the lemma only applies to sets systems containing many elements, not to individual configurations. For such individual configurations, much more is necessary to guarantee ``quasi-independence''. For example, as noted in \cite{extra}, we have $\lambda(\mathcal{E}_m \cap \mathcal{E}_n) \ll \lambda(\mathcal{E}_m) \lambda(\mathcal{E}_n)$ provided that $n \geq m^4$; in other words, to have quasi-independence for individual sets we need $n \geq m^4$ rather than $n \geq m (\log m)^c$, which obviously is a much stronger requirement.\\

As a consequence of Lemma \ref{lemmak1k2}, the real problem is to control the overlaps of sets from $S_{k_1}^{r_1}$ with sets from $S_{k_2}^{r_2}$ in the case when both $k_1$ and $k_2$ as well as $r_1$ and $r_2$ are very close to each other (that is, overlaps of sets $\mathcal{E}_m$ and $\mathcal{E}_n$ such that $m$ and $n$ are of comparable size, and $\psi(m)$ and $\psi(n)$ also are of comparable size). A particular instance of this problem is $k_1=k_2$ and $r_1=r_2$, i.e.\ when we try to control the overlap of sets from $S_{k}^r$ with other elements from $S_k^r$. We cannot control these overlaps in a way similar to Lemma \ref{lemmak1k2}, since that lemma relies on a Cauchy-Schwarz type estimate which only works when either $k_1-k_2$ or $r_1-r_2$ are large. However, we can introduce an additional omega-function (number of distinct prime divisors function) into the GCD sum estimate, and deduce that even if there might be many pairs of indices $m$ and $n$ for which the overlap $\mathcal{E}_m \cap \mathcal{E}_n$ is too large, then at least we can guarantee that typically such $m$ and $n$ cannot have too many distinct prime divisors, a fact which we can use to bound the function $P(m,n)$ as defined in \eqref{pdef}.

\begin{lemma} \label{lemma_h3}
There exists a constant $b_4>0$ such that the following holds. Let $\mathcal{M}$ denote a finite set of distinct positive integers, and write $N = \# \mathcal{M}$. Let $\omega(\cdot)$ denote the number of distinct prime factors of an integer. Then
$$
\sum_{m,n \in \mathcal{M}} \frac{(m,n)^{1/2}}{(mn)^{1/4}} 2^{\omega(mn/(m,n)^2)} \leq N^{3/2} (\log N)^{b_4},
$$
provided that $N$ is sufficiently large.
\end{lemma}

\begin{lemma} \label{lemma_diag}
Let $b_3$ be the constant from the statement of Lemma \ref{lemmak1k2}. There exist constants $b_5>0$ and $b_6>5$ such that the following holds. Let $S_{k_1}^{r_1}$ and $S_{k_2}^{r_2}$ be two sets as defined above, such that $k_1 \geq k_2$. Assume that
$$
|k_1 - k_2| \leq b_3 \log k_1 \qquad \textup{as well as} \qquad |r_1 - r_2| \leq b_3 \log k_1.
$$
Then the number of pairs of integers $m \in S_{k_1}^{r_1}$ and $n \in S_{k_2}^{r_2}$ for which both inequalities
\begin{equation} \label{both_1}
(m,n) \geq \frac{2^{k_1+1} 2^{-r_2}}{k_1^{b_3 + 1}}
\end{equation}
and
\begin{equation} \label{both_2}
\omega\left(\frac{mn}{(m,n)^2} \right) \geq b_5 \log k_1
\end{equation}
hold, is at most 
$$
2^{r_1} 2^{r_2} k_1^{-b_6},
$$
provided that $k_1$ and $k_2$ are sufficiently large.
\end{lemma}
The message of Lemma \ref{lemma_diag} is the following. Whenever $S_{k_1}^{r_1}$ and $S_{k_2}^{r_2}$ are such that $k_1$ and $k_2$ are close to each other, and such that $r_1$ and $r_2$ are also close to each other, then (in contrast to the situation of Lemma \ref{lemmak1k2}) we cannot rule out the possibility that there are many pairs $m \in S_{k_1}^{r_1}$ and $n \in S_{k_2}^{r_2}$ for which the overlaps $\mathcal{E}_m \cap \mathcal{E}_n$ are too large. However, even if there might exist many such pairs, then at least we can show that we might assume that such $m$ and $n$ have only few different prime factors. This allows you to obtain a more efficient estimate for the size of the overlaps, since the number of different prime factors enters the overlap estimate via the function $P(m,n)$ defined in \eqref{pmn}. On a quantitative level, note that in earlier work (such as \cite{a2, extra}) the number of different prime factors of $m$ and $n$ could only be estimated by $\ll (\log n)^{\varepsilon}$, whereas now we have the upper bound $\ll \log \log n$. This is where the gain in the ``extra divergence'' result comes from.\\

The remaining part of this paper is organized as follows. Section \ref{sec_lemma_proofs} contains the proofs of Lemmas \ref{lemma_h2} and \ref{lemmak1k2}, dealing with the overlap estimates in the case when at least one of $|k_1 - k_2|$ or $|r_1-r_2|$ is ``large''. Section \ref{sec_lemma_proofs_B} contains the proofs of Lemmas \ref{lemma_h3}--\ref{lemma_diag}, dealing with the overlap estimates in the case when both $|k_1-k_2|$ and $|r_1-r_2|$ are ``small''. Section \ref{sec_proof_th} contains the proof of Theorem \ref{th1}.

\section{Proof of Lemmas \ref{lemma_h2} and \ref{lemmak1k2}} \label{sec_lemma_proofs}

\begin{proof}[Proof of Lemma \ref{lemma_h2}]
We will use the fact that a GCD sum can be realized as an $L^2$-norm of a sum of dilated functions, a viewpoint which is also taken in e.g.\ \cite{a1,lind}. Let $f(x)$ be the function
$$
f(x) = \sum_{j \in \mathcal{J}} \frac{\sin(2 \pi j x)}{j^{1/4}},
$$
where $\mathcal{J}$ is the set
$$
\mathcal{J} := \{j \geq 1:~ 2^{k_1}/(R \log N^2) \leq j \leq 2^{k_1}/R \} \bigcup \left\{ j \geq 1:~ 2^{k_2}/(R \log N_2) \leq j \leq 2^{k_2}/R \right\}.
$$
Then by orthogonality
\begin{eqnarray*}
& & \int_0^1 \left(\sum_{m \in \mathcal{M}_1} f(mx) \right) \left(\sum_{n \in \mathcal{M}_2} f(nx) \right) ~dx \\
& = & \frac{1}{2} \sum_{\substack{j_1,j_2 \in \mathcal{J} \\ m \in \mathcal{M}_1, n \in \mathcal{M}_2,\\ j_1 m = j_2 n}} \frac{1}{j_1^{1/4} j_2^{1/4}}.
\end{eqnarray*}
The solutions of $j_1 m = j_2 n$ are of the form
\begin{equation} \label{h}
j_1 = h \frac{n}{(m,n)} , \qquad j_2 = h \frac{m}{(m,n)}, \qquad \textup{for some $h \in \mathbb{Z},~h \geq 1$}.
\end{equation}
By the construction of $\mathcal{J}$, for $h=1$ in \eqref{h} both numbers $j_1 = n/(m,n)$ and $j_2=m/(m,n)$ are contained in $\mathcal{J}$. Thus
\begin{equation} \label{gcd_int}
\int_0^1 \left(\sum_{m \in \mathcal{M}_1} f(mx) \right) \left(\sum_{n \in \mathcal{M}_2} f(nx) \right) ~dx \geq \frac{1}{2} \sum_{m \in \mathcal{M}_1, n \in \mathcal{M}_2} \frac{(m,n)^{1/2}}{(mn)^{1/4}}.
\end{equation}
On the other hand, by Cauchy-Schwarz,
\begin{eqnarray}
& & \left(\int_0^1 \left(\sum_{m \in \mathcal{M}_1} f(mx) \right) \left(\sum_{n \in \mathcal{M}_2} f(nx) \right) ~dx\right)^2 \\  \nonumber
& \leq & \left(\int_0^1 \left(\sum_{m \in \mathcal{M}_1} f(mx) \right)^2~dx \right) \left( \int_0^1 \left(\sum_{n \in \mathcal{M}_2} f(nx) \right)^2 ~dx\right). \label{second_int}
\end{eqnarray}
By expanding the square and using orthogonality we obtain
\begin{eqnarray*}
\int_0^1 \left(\sum_{m \in \mathcal{M}_1} f(mx) \right)^2~dx & = & \frac{1}{2} \sum_{\substack{j_1,j_2 \in \mathcal{J} \\ m,n \in \mathcal{M}_1,\\ j_1 m = j_2 n}} \frac{1}{j_1^{1/4} j_2^{1/4}}. 
\end{eqnarray*}
Again the solutions to $j_1 m = j_2 n$ are parametrized as in \eqref{h}, with one solution being $j_1 = m/(m,n)$ and $j_2 = n/(m,n)$, and the others being integer multiples. By the construction of the set $\mathcal{J}$, there can be at most $2(\log N_2)$ possible values in this parametrization. Thus
$$
\frac{1}{2} \sum_{\substack{j_1,j_2 \in \mathcal{J} \\ m,n \in \mathcal{M}_1,\\ j_1 m = j_2 n}} \frac{1}{j_1^{1/4} j_2^{1/4}} \leq 2 \log N_2 \sum_{m,n \in \mathcal{M}_1} \frac{(m,n)^{1/2}}{(mn)^{1/4}} \leq 2 N_1^{3/2} (\log N_1)^{b_1} (\log N_2),
$$
for sufficiently large $N_1$, as a consequence of Lemma \ref{lemma_h1}. Similarly, we estimate the second integral in \eqref{second_int}, and obtain
$$
\int_0^1 \left(\sum_{n \in \mathcal{M}_2} f(nx) \right)^2 ~dx \leq 2 N_2^{3/2} (\log N_2)^{b_1} (\log N_2)
$$
for sufficiently large $N_2$. Combining these estimates with \eqref{gcd_int} and \eqref{second_int} we obtain
$$
\frac{1}{2} \sum_{m \in \mathcal{M}_1, n \in \mathcal{M}_2} \frac{(m,n)^{1/2}}{(mn)^{1/4}} \leq 2 (N_1 N_2)^{3/4} (\log N_2)^{b_1+1}.
$$
This gives the conclusion of the lemma, if we choose for $b_2$ any number greater than $b_1+1$.
\end{proof}

\begin{proof}[Proof of Lemma \ref{lemmak1k2}]
Let $k_1,k_2,r_1,r_2$ be as in the statement of the lemma, and recall that we assumed $k_1 \geq k_2$. By \eqref{skr_lower} we have lower bounds on the cardinalities of $S_{k_1}^{r_1}$ and $S_{k_2}^{r_2}$, respectively. Recall how $D(m,n)$ was defined in \eqref{ddef}. By Lemma \eqref{lemmapv} and \eqref{pmn} for two sets $\mathcal{E}_m$ with $m \in S_{k_1}^{r_1}$, and $\mathcal{E}_n$ with $n \in S_{k_2}^{r_2}$, we have
\begin{equation} \label{empty_int}
\lambda (\mathcal{E}_m \cap \mathcal{E}_n) = \emptyset
\end{equation}
whenever
$$
D(m,n) < 1, 
$$
which necessarily happens whenever
$$
\frac{\max\{2^{k_1+1} 2^{-r_2},2^{k_2+1} 2^{-r_1} \}}{(m,n)} < 1,
$$
that is, whenever
$$
(m,n) > \max\{2^{k_1+1} 2^{-r_2},2^{k_2+1} 2^{-r_1} \}.
$$
Furthermore, as noted in \eqref{pmn}, we have
\begin{equation} \label{quasi-in}
\lambda (\mathcal{E}_m \cap \mathcal{E}_n) \ll \lambda (\mathcal{E}_m) \lambda(\mathcal{E}_n),
\end{equation}
unless 
$$
\frac{\max\{2^{k_1} 2^{-r_2},2^{k_2} 2^{-r_1} \}}{(m,n)} \leq k_1,
$$
which is equivalent to 
$$
(m,n) \geq \frac{\max\{2^{k_1} 2^{-r_2-1},2^{k_2} 2^{-r_1-1} \}}{k_1}.
$$
Thus the only critical case is when 
$$
\frac{1}{2} \max\{2^{k_1} 2^{-r_2},2^{k_2} 2^{-r_1} \} < (m,n) \leq 2 k_1 \max\{2^{k_1} 2^{-r_2},2^{k_2} 2^{-r_1} \}.
$$

We will split the proof of the lemma into two cases.\\

Case 1: 
\begin{equation} \label{add_ass}
|(k_1-r_2)-(k_2-r_1)| \geq \frac{b_3}{2} \log(k_1).
\end{equation}

Case 2: 
\begin{equation} \label{add_ass_2}
|(k_1-r_2)-(k_2-r_1)| \leq \frac{b_3}{2} \log(k_1).
\end{equation}

\begin{itemize}
\item Case 1: We assume that \eqref{add_ass} holds. Using Lemma \ref{lemma_h2} with $R = \frac{1}{2} \max\{2^{k_1} 2^{-r_2},2^{k_2} 2^{-r_1} \}$ we have
\begin{equation} \label{gcd_sum_case1}
\sum_{\substack{m \in S_{k_1}^{r_1},n \in S_{k_2}^{r_2},\\ R \leq (m,n) \leq R \log N_2}} \frac{(m,n)^{1/2}}{(mn)^{1/4}} \leq (\# S_{k_1}^{r_1} \cdot \# S_{k_2}^{r_2})^{3/4} (\log \max(\# S_{k_1}^{r_1}, \# S_{k_2}^{r_2})^{b_2}.
\end{equation}
Thus the number of pairs of indices $m \in S_{k_1}^{r_1}$ and $n \in S_{k_2}^{r_2}$ for which 
\begin{equation} \label{mnholds}
\frac{1}{2} \max\{2^{k_1} 2^{-r_2},2^{k_2} 2^{-r_1} \} > (m,n)
\end{equation}
is bounded by 
\begin{eqnarray}
& & \frac{(2^{k_1+1} 2^{k_2+1})^{1/4} \left(\# S_{k_1}^{r_1} \cdot \# S_{k_2}^{r_2} \right)^{3/4} (\log (\max(\# S_{k_1}^{r_1}, \# S_{k_2}^{r_2})^{b_2})}{ \left(\frac{1}{2} \max\{2^{k_1} 2^{-r_2},2^{k_2} 2^{-r_1} \} \right)^{1/2}} \nonumber\\
& \ll & \frac{\left(2^{k_1+k_2}\right)^{1/4} \left(2^{r_1+r_2} \right)^{3/4} ( \max\{r_1,r_2\})^{b_2} } { \left(\max\{2^{k_1} 2^{-r_2},2^{k_2} 2^{-r_1} \} \right)^{1/2}} \nonumber\\
& \ll & \frac{\left(2^{k_1-r_1+k_2-r_2}\right)^{1/4} 2^{r_1+r_2} k_1^{b_2} } { \left( \max\{2^{k_1} 2^{-r_2},2^{k_2} 2^{-r_1} \} \right)^{1/2}} \nonumber\\
& \ll & \left(2^{k_1-r_1+k_2-r_2 - 2 \max(k_1-r_2,k_2-r_1)}\right)^{1/4} 2^{r_1+r_2} k_1^{b_2} \nonumber\\
& = & \left(2^{\min(k_1-r_2,k_2-r_1) - \max(k_1-r_2,k_2-r_1)}\right)^{1/4} 2^{r_1+r_2} k_1^{b_2} \nonumber\\
& \ll &\left(2^{ - \frac{b_3}{2} \log k_1} \right)^{1/4} 2^{r_1+r_2} k_1^{b_2} \nonumber\\
& \ll & 2^{r_1+r_2} k_1^{-b_5}, \label{number_of_pairs}
\end{eqnarray}
where $b_5 = \frac{b_3 \log 2}{8} - b_2$ is a positive constant. Here we used \eqref{add_ass}, and the fact that we assumed $k_1 > k_2$.\\

Recall that, as a consequence of the lines following \eqref{empty_int}, we have $\mathcal{E}_m \cap \mathcal{E}_n = \emptyset$ whenever \eqref{mnholds} fails. Using \eqref{prod_ratio}, for all $m \in S_{k_1}^{r_1}$ and $n \in S_{k_2}^{r_2}$, we have
$$
\lambda(\mathcal{E}_m \cap \mathcal{E}_n) \ll (\log k_1) 2^{-r_1} 2^{-r_2}
$$
(recall that we assumed that $k_1 > k_2$, and that by construction $\log n \ll k_2$). The number of pairs of indices $m$ and $n$ that we have to take into account is estimated in \eqref{number_of_pairs}. Thus we obtain
\begin{eqnarray*}
\sum_{m \in S_{k_1}^{r_1}, n \in S_{k_2}^{r_2}} \lambda (\mathcal{E}_m \cap \mathcal{E}_n) & \ll & (\log k_1) 2^{-r_1} 2^{-r_2} 2^{r_1+r_2} k_1^{-b_5} \\
& \ll & (\log k_1) k_1^{-b_5} \\
& \ll & (\log k_1) k_1^{-b_5} r_1^2 \frac{2^{-r_1}}{r_1^2} 2^{r_1} r_2^2 \frac{2^{-r_2}}{r_2^2} 2^{r_2} \\ 
& \ll & (\log k_1) k_1^{-b_5} k_1^4~ \# S_{k_1}^{r_1} ~2^{-r_1}~ \# S_{k_2}^{r_2} ~2^{-r_2} \\
& \ll & k_1^{-b_6} \left( \sum_{m \in S_{k_1}^{r_1}} \lambda(\mathcal{E}_m) \right) \left( \sum_{n \in S_{k_2}^{r_2}} \lambda(\mathcal{E}_n) \right). 
\end{eqnarray*}
where we used \eqref{skr_lower}, as well as $r_1 \leq k_1$ and $r_2 \leq k_2 \leq k_1$, which is justified by the lines following \eqref{skr_def}. He $b_6$ is an appropriate \emph{positive} constant. The fact that $b_6$ can be chosen with a positive value follows from the assumption on the size of $b_3$ in the statement of the lemma, and the way how $b_6$ depends on $b_3$. This proves the conclusion of the lemma under the additional assumption 1. It remains to prove the lemma under the additional assumption 2.\\
 
\item Case 2: We assume that \eqref{add_ass_2} holds. By the assumption of the lemma we have either 
$$
|k_1 - k_2| \geq b_3 \log k_1 \qquad \textup{or} \qquad |r_1 - r_2| \geq b_3 \log k_1.
$$
Assuming that $k_1 - k_2| < b_3 \log k_1$ thus requires $|r_1 - r_2| \geq b_3 \log k_1$, which together with \eqref{add_ass_2} implies that
\begin{eqnarray*}
|k_1 - k_2| & = & |k_1 - r_2  + r_2 - k_2 - r_1 + r_1| \\
& \geq & |r_2  - r_1| - |(k_1 - r_2) - (k_2 - r_1)| \\
& \geq & \frac{b_3}{2} \log k_1.
\end{eqnarray*}
As a consequence, in Case 2 we always have
\begin{equation} \label{always_have}
|k_1 - k_2| \geq \frac{b_3}{2} \log k_1.
\end{equation}
In other words, in Case 2 the integers in the set $S_{k_1}^{r_1}$ are significantly larger than those in the set $S_{k_2}^{r_2}$, and we can use this fact to deduce the conclusion of the lemma. Indeed, assuming $m \in S_{k_1}^{r_1}$ and $n \in S_{k_2}^{r_2}$ and estimating 
$$
\frac{(m,n)^{1/2}}{(mn)^{1/4}} \leq \frac{n^{1/2}}{(mn)^{1/4}} \ll \frac{2^{k_2/2}}{2^{k_1/4 + k_2/4}},
$$
similar to \eqref{gcd_sum_case1} we obtain
\begin{eqnarray*}
& & \sum_{\substack{m \in S_{k_1}^{r_1},n \in S_{k_2}^{r_2},\\ R \leq (m,n) \leq R \log N_2}} \frac{(m,n)}{(mn)^{1/2}}  \\
& \ll & \frac{2^{k_2/2}}{2^{k_1/4 + k_2/4}} \sum_{\substack{m \in S_{k_1}^{r_1},n \in S_{k_2}^{r_2},\\ R \leq (m,n) \leq R \log N_2}} \frac{(m,n)^{1/2}}{(mn)^{1/4}} \\
& \ll & \frac{2^{k_2/2}}{2^{(k_1+k_2)/4}} 2^{3(r_1 + r_2)/4} k_1^{b_2}.
\end{eqnarray*}

Thus the number of pairs of indices $m \in S_{k_1}^{r_1}$ and $n \in S_{k_2}^{r_2}$ for which \eqref{mnholds} holds is bounded by 
\begin{eqnarray}
& \ll & \frac{2^{(k_1+k_2)/4} 2^{k_2/2} 2^{3(r_1 + r_2)/4} k_1^{b_2}}{ 2^{(k_1+k_2)/4} \left(\max\{2^{k_1} 2^{-r_2},2^{k_2} 2^{-r_1} \} \right)^{1/2}} \nonumber\\
& \ll & \frac{2^{k_2/2} 2^{3(r_1 + r_2)/4} k_1^{b_2}}{\left(2^{k_1} 2^{-r_2} 2^{k_2} 2^{-r_1}\right)^{1/4}} \nonumber\\
& \ll & 2^{r_1+r_2} 2^{(k_2-k_1)/4} k^{b_2} \nonumber\\
& \ll & 2^{r_1+r_2} k^{-b_5}, \nonumber
\end{eqnarray}
where as above $b_5 = \frac{b_3 \log 2}{b_2}$. Thus we have obtained an estimate similar to \eqref{number_of_pairs}, and thus the proof in Case 2 can be concluded in the same way as the proof in Case 1. Thus we have established Lemma \ref{lemmak1k2}.

\end{itemize}

\end{proof}

\section{Proofs of Lemmas \ref{lemma_h3} -- \ref{lemma_diag}} \label{sec_lemma_proofs_B}

\begin{proof}[Proof of Lemma \ref{lemma_h3}]
The proof of Lemma \ref{lemma_h3} can be given following the arguments in \cite{hilb}. The omega-function already appears there in Lemma 1, where the authors pass from the square-free case to the general situation, so in principle the ground is prepared for introducing the omega-function into the estimate for the GCD sum. To find the omega-function in the general GCD sum estimate as in our Lemma \ref{lemma_h1}, one has to prove a version of \cite[Lemma 1]{hilb} with the factor $2^{\omega(mn/(m,n)^2)}$ replaced by $4^{\omega(mn/(m,n)^2)}$. This requires only some minor modifications in the proof given in \cite{hilb}, so we just indicate what modifications are necessary there. Note that their result has a parameter $\alpha \in (0,1/2)$, to which we assign the special value $\alpha=1/4$.\\ 

Let $\mathcal{M}$ denote a set of distinct positive integers, and write $N = |\mathcal{M}|$. According to \cite[Lemma 1]{hilb} there exists a divisor-closed set $\mathcal{M}'$, also of cardinality $|\mathcal{M}'| = N$, such that
$$
\sum_{m,n \in \mathcal{M}} \frac{(m,n)^{1/2}}{(mn)^{1/4}} \leq \sum_{m,n \in \mathcal{M}'} \frac{(m,n)^{1/2}}{(mn)^{1/4}} 2^{\omega(mn/(m,n)^2)}.
$$
Here the term ``divisor-closed'' means that whenever a positive integer is contained in $\mathcal{M}'$, then all its divisors are contained in $\mathcal{M}'$ as well. Thus for the quantity that we want to estimate in Lemma \ref{lemma_h3}, we have
$$
\sum_{m,n \in \mathcal{M}} \frac{(m,n)^{1/2}}{(mn)^{1/4}} 2^{\omega(mn/(m,n)^2)} \leq \sum_{m,n \in \mathcal{M}'} \frac{(m,n)^{1/2}}{(mn)^{1/4}} 4^{\omega(mn/(m,n)^2)}.
$$

Thus let now $\mathcal{M}$ be any divisor-closed set of distinct positive integers such that $|\mathcal{M}|= N$. Our aim is to prove an upper bound for 
$$
\sum_{m,n \in \mathcal{M}} \frac{(m,n)^{1/2}}{(mn)^{1/4}} 4^{\omega(mn/(m,n)^2)}.
$$
We can follow the proof given on pages 99-100 of \cite{hilb} verbatim line by line, with a very few exceptions. The only necessary modifications are: everywhere in the proof, the terms $2^{\omega(mn/(m,n)^2)}$ and $2^{\omega(kl/(k,\ell)^2)}$ have to be replaced by $4^{\omega(mn/(m,n)^2)}$ and $4^{\omega(k \ell/(k,\ell)^2)}$, respectively. In line 9 from below on page 99 of \cite{hilb}, the term $4 \sum_{\nu=1}^\infty p^{-\nu}$ has to be changed to $8 \sum_{\nu=1}^\infty p^{-\nu}$. This does not lead to any change in the following line, since only the value of the implied constant in the ``$\ll$'' symbol is affected. In the last displayed formula on page 99, the term $d(m)^{1+\varepsilon}$ has to be changed to $d(m)^{2+\varepsilon}$, since for a square-free number $m$ instead of $2^{\omega(m)} = d(m)$ we now have to use $4^{\omega(m)} = d(m)^2$. On the first displayed formula of page 100 of \cite{hilb}, we consequently have to replace the term $d(m)^{\beta+2+4\varepsilon}$ by $d(m)^{\beta+4+4\varepsilon}$, which means that later in the proof we have to choose $\beta' > (\beta+4+4\varepsilon)/(2\alpha)$ rather than $\beta' > (\beta+2+4\varepsilon)/(2\alpha)$. Everything else remains completely unchanged, and we otain the same conclusion as in \cite{hilb}, only with a different exponent of the logarithmic term.

\end{proof}

\begin{proof}[Proof of Lemma \ref{lemma_diag}]
Let $k_1 \geq k_2$ be as in the statement of the lemma. We use Lemma \ref{lemma_h3} for the set $\mathcal{M} = S_{k_1}^{r_1} \cup S_{k_2}^{r_2}$. Recall that by \eqref{skr_size} we can assume that $\# S_{k_1}^{r_1} \leq k_1 2^{r_1}$ and $\# S_{k_2}^{r_2} \leq k_2 2^r_2 \leq k_1 2^r_2$. Applying Lemma \ref{lemma_h3}, and using that $|r_1 - r_2| \leq b_3 \log k_1$ implies $2^{r_1} \ll 2^{r_2} k^{b_3 \log 2}$, we obtain
\begin{eqnarray*}
\sum_{m,n \in \mathcal{M}} \frac{(m,n)^{1/2}}{(mn)^{1/4}} 2^{\omega(mn/(m,n)^2)} & \ll &  N^{3/2} (\log N)^{b_4} \\
& \ll & (k_1 2^{\max(r_1,r_2)})^{3/2} k_1^{b_4} \\
& = & 2^{3r_2/2} k_1^{b_4+3/2+(3b_3\log 2)/2}.
\end{eqnarray*}
Note that for $m,n \in \mathcal{M}$ the inequalities \eqref{both_1} and \eqref{both_2} require that
$$
(mn)^{1/4} \gg 2^{k_1/2}, \qquad (m,n) \gg \frac{2^{k_1} 2^{-r_2}}{k_1^{b_3 + 1}}, \qquad 2^{\omega(mn/(m,n)^2)} \gg k_1^{b_5 \log 2}.
$$
Thus for the number of pairs of indices $m,n \in \mathcal{M}$ for which both estimates \eqref{both_1} and \eqref{both_2} hold is at most
\begin{eqnarray*}
& \ll & 2^{3r_2/2} k_1^{b_4+3/2+(3b_3\log 2)/2} ~\frac{2^{k_1/2} 2^{r_2/2} k_1^{b_3/2 + 1/2}}{2^{k_1/2} k_1^{b_5 \log 2}} \\
& \ll & 2^{2r_2} k_1^{b_4+2+b_3/2+(3b_3\log 2)/2-b_5 \log 2} \\
& \ll & 2^{r_1} 2^{r_2} k_1^{b_4+2+b_3/2+(4b_3\log 2)/2-b_5 \log 2}.
\end{eqnarray*}
Let $b_6=-(b_4+2+b_3/2+(4b_3\log 2)/2-b_5 \log 2)$. If $b_5$ is chosen sufficiently large, then we can assume that $b_6>5$, and for the number of pairs satisfying \eqref{both_1} and \eqref{both_2} we have the upper bound $2^{r_1} 2^{r_2} k_1^{-b_6}$, as desired.
\end{proof}

\section{Proof of Theorem \ref{th1}} \label{sec_proof_th}

The proof of Theorem \ref{th1} follows the same path as the one given in \cite{a2}. The two new ingredients that we have are that
\begin{itemize}
 \item a) whenever $m$ and $n$ are such that either $m/n$ or $\psi(m)/\psi(n)$ are bounded away from 1, by a factor of order at least $\log \log n$, then we can control the size of the overlap $\mathcal{E}_m \cap \mathcal{E}_n$, and
 \item b) whenever $m$ and $n$ are such that $m/n$ and $\psi(m)/\psi(n)$ are both very close to 1, we may assume that $m$ and $n$ have at most $\ll \log \log n$ different prime factors.
\end{itemize}

\vspace{.5cm}
Throughout the proof, let $\varepsilon>0$ be fixed. Let the function $\psi: \mathbb{N} \mapsto [0,1/2]$ be given, and assume that \eqref{div} holds. We define the sets $S_k^r$ as in \eqref{skr_def}. As noted we may restrict $k$ and $r$ to the range $1 \leq r \leq k$. We may also assume throughout the proof that
$$
S_k^r \geq \frac{2^r}{r^2}, 
$$
since otherwise the indices in $S_k^r$ do not contribute to the divergence of the series \eqref{div}.\\

We split the positive integers into blocks of the form $\{2^{4^h}+1, 2^{4^{h+1}}\}$. As argued in \cite{extra}, we may assume that $\psi$ in supported only on integers that are contained such blocks for \emph{even} values of $h$. Also, again following \cite{extra}, we only need to control the overlaps of sets with indices $m$ and $n$ that are contained in the same block of this form; whenever $m$ and $n$ come from blocks with different values of $h$, then the corresponding sets are automatically quasi-independent.\\

We fix a positive integer $h$, and we assume that $h$ is ``large''. Let $S = S(h) = \lfloor \varepsilon \log h \rfloor$ . For every $n \in [2^{4^h}, 2^{4^{h+1}})$ and for every $s \in \{1, \dots, S\}$ we define sets $\mathcal{E}_n^{(s)}$ in a way similar to the definition of $\mathcal{E}_n$, but with $\psi(n)/e^s$ instead of $\psi(n)$. That is, 
\begin{equation} \label{ens}
\mathcal{E}_n^{(s)} = \bigcup_{\substack{1 \leq a \leq n,\\ (a,n)=1}} \left( \frac{a - \psi(n)}{n e^s},\frac{a+\psi(n)}{n e^s} \right).
\end{equation}
We emphasize that in all the estimates that follow, the implied constant in the symbol ``$\ll$'' does not depend on the value of $s \in \{1, \dots, S(h)\}$.\\ 

For every $s$ we have
\begin{eqnarray}
& & \sum_{2^{4^h} \leq m,n \leq 2^{4^{h+1}}} \lambda \left(\mathcal{E}_m^{(s)} \cap \mathcal{E}_n^{(s)} \right) \nonumber\\
& = & \sum_{1 \leq k_1,k_2 \leq K} \sum_{r_1 = 1}^{k_1} \sum_{r_2=1}^{k_2} \sum_{\substack{m \in S_{k_1}^{r_1}, ~n \in S_{k_2}^{r_2}}} \lambda \left(\mathcal{E}_m^{(s)} \cap \mathcal{E}_n^{(s)} \right). \label{every_s}
\end{eqnarray}
Lemma \ref{lemmak1k2} and Lemma \ref{lemma_diag} remain true without any change if the sets $\mathcal{E}_m$ and $\mathcal{E}_n$ are replaced by $\mathcal{E}_m^{(s)}$ and $\mathcal{E}_n^{(s)}$, respectively, since both lemmas only depend on the relative positions of $k_1$ and $k_2$ (which remain unchanged when passing from the sets $\mathcal{E}$ to the sets $\mathcal{E}^{(s)}$) and $r_1$ and $r_2$ (both of which are shifted in the same way when changing from $\mathcal{E}$ to $\mathcal{E}^{(s)}$). Thus by Lemma \ref{lemmak1k2} we have
\begin{equation} \label{we_have}
{\sum}^* ~\sum_{\substack{m \in S_{k_1}^{r_1},~n \in S_{k_2}^{r_2}}} \lambda \left(\mathcal{E}_m^{(s)} \cap \mathcal{E}_n^{(s)} \right) \leq \left(\sum_{2^{4^h} \leq n \leq 2^{4^{h+1}}} \lambda \left(\mathcal{E}_n^{(s)}\right) \right)^2.
\end{equation}
where the sum with the asterisk extends over all values of $k_1,k_2,r_1,r_2$ in the range $1 \leq k_1,k_2 \leq K, ~1 \leq r_1 \leq k_1, ~ 1\leq r_2 \leq k_2$ for which additionally either  
$$
|k_1-k_2| \geq b_3 \log k_1 \quad \text{ or } \quad |r_1-r_2| \geq b_3 \log k_1 \quad \text{holds}.
$$
So it remains to control the overlaps of sets $\mathcal{E}_m^{(s)}$ and $\mathcal{E}_n^{(s)}$, where $m \in S_{k_1}^{r_1}$ and $n \in S_{k_2}^{r_2}$,  in the case when both $|k_1 - k_2|$ and $|r_1-r_2|$ are small. Thus let $k_1,k_2,r_1,r_2$ be given, and assume that 
$$
|k_1-k_2| \geq b_3 \log k_1 \quad \text{ and } \quad |r_1-r_2| \geq b_3 \log k_1.
$$
Furtermore, w.l.o.g.\ we assume that $k_1 \geq k_2$. Let $m \in S_{k_1}^{r_1}$ and $n \in S_{k_2}^{r_2}$. As in the lines around \eqref{empty_int}, we have
$$
\mathcal{E}_m^{(s)} \cap \mathcal{E}_n^{(s)} = \emptyset, 
$$
whenever
\begin{equation} \label{whenever}
\max \left\{ \frac{n \psi(m)}{(m,n) e^s}, \frac{m \psi(n)}{(m,n) e^s} \right\} < 1.
\end{equation}
Note that 
$$
m \psi(n) \ll 2^{k_1} 2^{-r_2},
$$
and by the assumptions on $k_1,k_2,r_1,r_2$ we have
$$
n \psi(m) \ll 2^{k_2} 2^{-r_1} \ll 2^{k_1} 2^{-r_1} \ll 2^{k_1} 2^{-r_2} 2^{b_3 \log k_1} \ll 2^{k_1} 2^{-r_1} k_1^{b_3}.
$$
Furthermore, we have
$$
e^s \leq e^{S(h)} \ll e^{\varepsilon \log h} \ll e^{\varepsilon \log k_1} \ll k_1^\varepsilon
$$
for sufficiently large $h$. Thus for sufficiently large $h$ the inequality \eqref{whenever} is always satisfied when
\begin{equation} \label{when_sat}
(m,n) < \frac{2^{k_1+1} 2^{-r_2}}{k_1^{b_3 +1}}.
\end{equation}
Thus we have shown that, independent of the value of $s \in \{1, \dots, S(h)\}$, we have
$$
\mathcal{E}_m^{(s)} \cap \mathcal{E}_n^{(s)} = \emptyset
$$
whenever \eqref{when_sat} holds. In other words, it is sufficient to consider only those pairs of $m$ and $n$ for which \eqref{when_sat} fails. Thus for every $s$ we have
\begin{eqnarray}
\sum_{\substack{m \in S_{k_1}^{r_1},~n \in S_{k_2}^{r_2}}} \lambda \left(\mathcal{E}_m^{(s)} \cap \mathcal{E}_n^{(s)} \right) & = & {\sum}^{(1)} \lambda \left(\mathcal{E}_m^{(s)} \cap \mathcal{E}_n^{(s)} \right) + {\sum}^{(2)} \lambda \left(\mathcal{E}_m^{(s)} \cap \mathcal{E}_n^{(s)} \right). \label{every_s_b}
\end{eqnarray}
Here ${\sum}^{(1)}$ extends over those $m \in S_{k_1}^{r_1}$ and $n \in S_{k_2}^{r_2}$ for which
\begin{equation} \label{both_both}
(m,n) < \frac{2^{k_1+1} 2^{-r_2}}{k_1^{b_3 +1}} \quad  \text{ and } \quad \omega\left(\frac{mn}{(m,n)^2} \right) \geq b_5 \log k_1,
\end{equation}
while ${\sum}^{(2)}$ extends over those $m \in S_{k_1}^{r_1}$ and $n \in S_{k_2}^{r_2}$ for which
\begin{equation} \label{both_both_both}
(m,n) < \frac{2^{k_1+1} 2^{-r_2}}{k_1^{b_3 +1}} \quad  \text{ and } \quad \omega\left(\frac{mn}{(m,n)^2} \right) < b_5 \log k_1.
\end{equation}
By Lemma \ref{lemma_diag} the cardinality of the sets of pairs $m$ and $n$ with $m \in S_{k_1}^{r_1}$ and $n \in S_{k_2}^{r_2}$ for which both inequalities in \eqref{both_both} holds is of order at most $2^{r_1} 2^{r_2} k_1^{-b_6}$. Recall that 
$$
\lambda \left( \mathcal{E}_m^{(s)} \cap \mathcal{E}_n^{(s)} \right) \ll \lambda(\mathcal{E}_m^{(s)}) \lambda (\mathcal{E}_n^{(s)} ) \log k_1,
$$
as a consequence of \eqref{prod_ratio}. Thus we have
\begin{eqnarray}
{\sum}^{(1)} \lambda \left(\mathcal{E}_m^{(s)} \cap \mathcal{E}_n^{(s)} \right) & \ll & 2^{r_1} 2^{r_2} k_1^{-b_6} 2^{-r_1} 2^{-r_2}  e^{-2s} \log k_1 \nonumber\\
& \ll & 2^{r_1} 2^{r_2} e^{-2s} k_1^{-5} \nonumber\\
& \ll & |S_{k_1}^{r_1}| 2^{-r_1} |S_{k_2}^{r_2}| 2^{-r_2} e^{-2s} k^{-1} \nonumber\\
& \ll & \left( \sum_{m \in S_{k_1}^{r_1}} \lambda(\mathcal{E}_m^{(s)} \right) \left( \sum_{n \in S_{k_2}^{r_2}} \lambda(E_n^{(s)} \right). \label{sum1_desired}
\end{eqnarray}
Here we used \eqref{skr_lower} and the fact that $b_6 > 5$. Thus the contribution of ${\sum}^{(1)}$ is bounded in the desired way, and the remaining problem is to control ${\sum}^{(2)}$. This is where we apply the well-known averaging procedure from \cite{a2, extra}, together with the new ingredient that by \eqref{both_both_both} we only need to consider pairs $m$ and $n$ which have a limited number of distinct prime factors.\\

Let $m$ and $n$ be a pair of indices which contributes to the sum ${\sum}^{(1)}$. For $s \in \{1,\dots, S(h)\}$, we set 
$$
P_s (m,n) = 	\prod_{\substack{p \mid \frac{mn}{(m,n)^2},\\ p > e^s}} \left( 1 - \frac{1}{p} \right)^{-1}.
$$
We note that we can completely ignore the contribution of all primes $p > e^{S(h)}$, as a consequence of the assumption $\omega\left(\frac{mn}{(m,n)^2} \right) < b_5 \log k_1$ and Mertens' theorem. Indeed, we have
\begin{eqnarray*}
\log \left( \prod_{\substack{p \mid \frac{mn}{(m,n)^2},\\ p > e^{S(h)}}} \left( 1 - \frac{1}{p} \right)^{-1} \right)
& \ll & \sum_{\substack{p \mid \frac{mn}{(m,n)^2},\\ p > e^{S(h)}}} \frac{1}{p}  \\
& \ll & {\sum}^* \frac{1}{p},
\end{eqnarray*}
where the summation in ${\sum}^*$ extends over the $b_5 \log k_1$ smallest primes exceeding $e^{S(h)}$. For this sum we have
\begin{eqnarray*}
{\sum}^* \frac{1}{p} & \ll & (\log \log (e^{S(h)} + \log k_1)) - \log \log \log e^{S(h)} \\
& \ll & (\log \log \log h) - \log \log (\varepsilon \log h) \\
& \ll & 1,
\end{eqnarray*}
since $\varepsilon$ is assumed to be fixed. Thus the contribution of ```large'' primes can be ignored, which allows us to use a shorter summation range for the factors $P_s$ than in \cite{a2}. Following the lines in \cite{a2}, we can now prove that
$$
\sum_{s=1}^{S(h)} P_s(m,n) \ll S(h).
$$
Thus together with \eqref{every_s}, \eqref{we_have}, \eqref{every_s_b} and \eqref{sum1_desired} we have
\begin{eqnarray*}
& & \sum_{s=1}^{S(h)} \sum_{2^{4^h} \leq m,n \leq 2^{4^{h+1}}} \lambda \left(\mathcal{E}_m^{(s)} \cap \mathcal{E}_n^{(s)} \right) \nonumber\\
& \ll & S(h) \left(\sum_{2^{4^h} \leq n \leq 2^{4^{h+1}}} \left(\mathcal{E}_m^{(s)} \right) \right)^2.
\end{eqnarray*}
As a consequence, there is a choice of $s \in \{1, \dots, S(h)\}$ such that
$$
\sum_{2^{4^h} \leq m,n \leq 2^{4^{h+1}}} \lambda \left(\mathcal{E}_m^{(s)} \cap \mathcal{E}_n^{(s)} \right) \ll \left(\sum_{2^{4^h} \leq n \leq 2^{4^{h+1}}} \left(\mathcal{E}_m^{(s)} \right) \right)^2.
$$
With this choice of $s$, we replace the function $\psi(n)$ by $\psi^*(n) = \psi(n)/e^s$ for all $n$ in the range $2^{4^h} \leq m,n \leq 2^{4^{h+1}}$, and we write $\mathcal{E}_n^*$ for the corresponding sets which are defined as in \eqref{ens} with this choice of $s$. Note that by our choice of $S(h)$, we have
$$
\psi^*(n) \gg \frac{\psi(n)}{(\log \log n)^{\varepsilon}}. 
$$
Thus from \eqref{div} we have
$$
\sum_{n=1}^\infty \lambda(\mathcal{E}_n^*) = \sum_{n=1}^\infty \frac{2 \psi^*(n) n}{\varphi(n)} \gg \sum_{n=1}^\infty \frac{\psi(n) n}{\varphi(n) (\log \log n)^{\varepsilon}} = \infty.
$$
By construction the sets $\mathcal{E}_n^*$ are quasi-independent, and thus by Lemma \ref{lemmabc} the set of those $x$ which are contained in infinitely many sets $\mathcal{E}_n^*$ has positive measure. By Gallagher's \cite{gall} zero-one law, positive measure implies full measure. Thus almost all $x \in [0,1]$ are contained in infinitely many sets $\mathcal{E}_n^*$. Since $\mathcal{E}_n^* \subset \mathcal{E}_n$, almost all $x \in [0,1]$ are contained in infinitely many sets $\mathcal{E}_n$. This proves the theorem.

\section{Proofs of Theorem \ref{th2} and \ref{th3}}  \label{sec:proofth2th3}

As noted after the statement of theorems, Theorem \ref{th2} follows directly from Theorem \ref{th3}. Thus we only have to prove Theorem \ref{th3}. The proof can be given in the spirit of the one in \cite{aistslow}, using the decoupling lemmas in this paper to obtain the improved result.

\section*{Acknowledgements}

The author is supported by the Austrian Science Fund (FWF), projects F-5512, I-3466 and Y-901. Thanks to the members of the Fufu seminar: Laima Kaziulyte, Thomas Lachmann, Marc Munsch, Niclas Technau and Agamemnon Zafeiropoulos.


\end{document}